\documentclass{amsart}

\usepackage{amssymb}

\usepackage{enumerate}

\DeclareMathOperator{\BV}{BV}
\DeclareMathOperator{\Ren}{Re}
\DeclareMathOperator{\Imn}{Im}

\newcommand{\abs}[1]{\mathord{\left|#1\right|}}
\newcommand{\zsum}{\sideset{}{^*}\sum}
\newcommand{\eps}{\varepsilon}
\newcommand{\half}{\frac12}
\newcommand{\R}{\mathbb{R}}
\newcommand{\C}{\mathbb{C}}
\newcommand{\N}{\mathbb{N}}
\newcommand{\Z}{\mathbb{Z}}
\newcommand{\Q}{\mathbb{Q}}
\newcommand{\mfa}{\mathfrak a}

\newcommand{\mff}{\mathfrak f}
\newcommand{\mfp}{\mathfrak p}
\newcommand{\mcL}{\mathcal L}
\newcommand{\mcE}{\mathcal E}
\newcommand{\mcO}{\mathcal O}

\newtheorem{theorem}{Theorem}[section]
\newtheorem{lemma}[theorem]{Lemma}

\theoremstyle{definition}
\newtheorem{definition}[theorem]{Definition}

\theoremstyle{remark}
\newtheorem{remark}[theorem]{Remark}

\numberwithin{equation}{section}

\begin{document}

\title[Proving the completeness of a list of zeros]{A method for proving the completeness of a list of zeros of certain L-functions}

\author{Jan B\"uthe}
\address{Mathematisches Institut, Bonn University, Endenicher Allee 60, 53115 Bonn, Germany}
\email{jbuethe@math.uni-bonn.de}


\date{\today}

\begin{abstract}
When it comes to partial numerical verification of the Riemann Hypothesis, one crucial part is to verify the completeness of a list of pre-computed zeros. Turing developed such a method, based on an explicit version of a theorem of Littlewood on the average of the argument of the Riemann zeta function. In a previous paper \cite{BFJK13} we suggested an alternative method based on the Weil-Barner explicit formula. This method asymptotically sacrifices fewer zeros in order to prove the completeness of a list of zeros with imaginary part in a given interval. In this paper, we prove a general version of this method for an extension of the Selberg class including Hecke and Artin L-series, L-functions of modular forms, and, at least in the unramified case, automorphic L-functions. As an example, we further specify this method for Hecke L-series and L-functions of elliptic curves over the rational numbers.

\end{abstract}

\maketitle

\section{Introduction}

In this paper we develop a general method to prove that a list of zeros of an L-function contains all zeros with imaginary part in a given interval $[a,b]$. The method is proved for all L-functions in an extension of the Selberg class.

The method is an alternative to the Turing method \cite{Turing1953,Lehman70}, of which a similar generic version has been developed in \cite{Booker06}. The methods are similar in that only the zeros in a neighbourhood of $a$ and $b$ have to be known within a higher accuracy, but they differ with respect to the required number of additional zeros with imaginary part outside of $[a,b]$. While the Turing method requires $O(\log(a)^2)$ additional zeros below $a$ and $O(\log(b)^2)$ additional zeros above $b$, the method described in this paper requires only $O(\log(a)\log\log(a))$, respectively, $O(\log(b)\log\log(b))$ such zeros. 

In the special case of the Riemann zeta function the implied constant for the Turing method is small due to sophisticated explicit estimates of the Riemann zeta function in the critical strip \cite{Trudgian2011B}. This led to the assumption in \cite{BFJK13} that the implied cross-over for the two methods would not occur before $10^{30}$. However, numerical tests suggest that the new method generally requires fewer additional zeros. Furthermore, we will provide improved estimates in this paper which show that at the moderate height of $10^6$ the Turing method already needs about twice as many additional zeros.

A second advantage of this method is, that the proof of completeness only depends on the correctness of the Weil-Barner explicit formula and very few explicit estimates of the test function to which the Weil-Barner formula is applied. Only the specification of the implied constant in the aforementioned $O$-terms requires longer calculations.

We will first formulate the explicit formula for the class of L-functions in consideration, and then prove the general version of the method. Then, we further specify the method for the Riemann zeta function, improving the results in \cite{BFJK13}, Hecke L-series and elliptic curves over $\Q$. As far as the author knows, the Turing method has not been adapted to the latter families of L-functions before.

\nocite{Rumley93,Tollis97}

\section{The Weil-Barner explicit formula}
The method is based on Barner's version \cite{Barner1981} of  Weil's explicit formula \cite{Weil1952}. Barner proves this explicit formula for Hecke L-series only, so we take the very general work of Jorgenson and Lang \cite{JL1994} as a reference. The explicit formula in \cite{JL1994}  covers exotic L-functions as the Selberg zeta function, which makes its proof long and complicated. Therefore, it should be pointed out, that the explicit formula for the L-functions considered in this paper can be proved more quickly by straightforward modifications of Barner's proof.

\nocite{Weil1972,Moreno1977}

We will consider all L-functions $L:\C\rightarrow \C\cup \{\infty\}$ satisfying the following properties:

\begin{enumerate}[(L1)]
\item There are numbers $a(n)\in\C$ and a constant $K\in\R$ such that we have
\[
 L(s) = \sum_{n=1}^\infty a(n) n^{-s}
\]
for $\Ren(s)>K$ where the sum converges absolutely.\label{L:Dirichlet-series}

\item There is a polynomial $P\in \C[x]$ such that the function $P(s)L(s)$ continues to an entire function of finite (power) order.\label{L:meromorphic-continuation}

\item There exist $c(p^m)$ for all prime powers $p^m$, a $\sigma_1\in\R$, and a $C>0$ such that we have
\[
\abs{c(p^m)} \leq C p^{(\sigma_1-1)m},
\]
and
\[
 L(s) = \exp\Bigl(\sum_{p^m} c(p^m) p^{-ms}\Bigr)
\]
for $\Ren(s) > \sigma_1$.\label{L:Euler-product}

\item There exists a $\sigma_0<2\sigma_1$, a $k_1\in\N$, $Q, \lambda_1,\dots, \lambda_k \in(0,\infty)$, $\mu_1,\dots, \mu_k\in \C$ with $\Ren(\mu_k)>-\lambda_k\sigma_0/2$, and a $w\in\C$ with $\abs w = 1$, such that the complete L-function 
\[
\Lambda(s) = Q^s \prod_{k=1}^{k_1} \Gamma(\lambda_ks + \mu_k)L(s)
\]
satisfies the functional equation
\[
 \Lambda(s) = w \overline{\Lambda(\sigma_0-\overline{s})}.
\]
\label{L:functional-equation}
\end{enumerate}

Compared to the Selberg class \cite{Selberg1989}, we essentially gave up the Ramanujan-Petersson conjecture, which for most automorphic L-functions is
still an open problem, and we allow a finite number of poles (not only at $s=1$). We also do not assume the L-function to be normalized, i.e. we allow $\sigma_0\neq 1$.

We use the convention $\mu_k = u_k + iv_k$, with $u_k$ and $v_k$ real, and define
\[
 G(s) = Q^s \prod_{k=1}^{k_1} \Gamma(\lambda_ks + \mu_k)
\]
and
\[
G_0(s) = Q^s \prod_{k=1}^{k_1} \Gamma(\lambda_ks + u_k).
\]

Furthermore, we denote by $\mathcal N(\Lambda)$ the set of zeros of $\Lambda$, which we also refer to as the non-trivial zeros of $L$, and by $\mathcal P(\Lambda)$ the set of poles.

Given the parameters $\sigma_0$ and $\sigma_1$, we define the Barner class of test functions (which actually only depends on $\sigma_1-\frac{\sigma_0}2$) to be the class of functions $f:\R\rightarrow \C$ satisfying the following properties:

\begin{enumerate}[({Ba}1)]
 \item There exists a $c> \sigma_1- \frac{\sigma_0}{2}$, such that
\[
 f(t) \exp(c\abs{t}) \in \BV(\R) \cap L^1(\R).
\]
\item The function $f$ is normalized, i.e., for every $t\in\R$ we have
\[
 f(t) = \frac12\lim_{h\searrow 0} \Bigl(f(t+h) + f(t-h)\Bigr).
\]
\item There exists an $\eps > 0$ such that we have
\[
 2 f(0) = f(t) + f(-t) + O(\abs{t}^\eps)
\]
for $t\to 0$.
\end{enumerate}

Now let $f$ be a  member of the Barner class. We define the Fourier transform by
\[
 \hat f(\xi) = \int_{-\infty}^\infty e^{i\xi t} f(t)\, dt.
\]
Then we have a pointwise Fourier inversion formula given by
\[
 f(t) = \frac{1}{2\pi}\int_{-\infty}^\infty e^{-i\xi t} \hat f(\xi)\, d\xi.
\]
With this we define the functionals
\begin{equation*}
 w_s(\hat f) = \zsum_{\rho} n_\rho  \hat f\bigl(\tfrac{\rho}{i}-\tfrac{\sigma_0}{2i}\bigr) 
= \lim_{T\to \infty} \sum_{\abs{\Imn(\rho)}<T} n_\rho  \hat f\bigl(\tfrac{\rho}{i}-\tfrac{\sigma_0}{2i}\bigr),
\end{equation*}
where the sum is taken over all zeros and poles of the function $\Lambda(s)$
according to their multiplicities $n_\rho$ (poles being counted with negative multiplicity),
\begin{equation*}
 w_f(f) = - 
\sum_{p^m} \frac{\log p}{\sqrt{p^{m\sigma_0}}}\Bigl(c(p^m) f(m\log p) + \overline{c(p^m)} f(-m\log p)\Bigr),
\end{equation*}
and
\begin{multline*}
 w_\infty(f) = 2f(0)\Ren\frac{G_0'}{G_0}\bigl(\frac{\sigma_0}{2}\bigr)  \\
 - \sum_{k=1}^n \lambda_k \int_{0}^\infty \frac{e^{-(\lambda_k \sigma_0/2 + u_k)t}}{1-e^{-t}}
\bigl(e^{-iv_kt}f(\lambda_kt) + e^{iv_kt}f(-\lambda_kt) - 2f(0)\bigr)\, dt.
\end{multline*}
Then, the Weil-Barner formula takes the form
\begin{equation*}
 w_s(\hat f) = w_f(f) + w_\infty(f).
\end{equation*}

\section{The Method}
Since the L-functions in consideration may have multiple zeros, we begin with the following definition.

\begin{definition}
By a \emph{list of zeros} of a function $f$ we shall mean a sequence $(\rho_j)_{j\in I}$ of zeros of $f$, such that every zero $\rho$ with multiplicity $n_\rho$ occurs at most $n_\rho$ times in the sequence. If $Z$ is a set of zeros of $f,$ we say that a list \emph{contains all zeros in $Z$} if every $\rho\in Z$ occurs exactly $n_\rho$ times.
\end{definition}

We aim to verify that a list of zeros of an L-function contains all zeros with imaginary part in $[a,b]$ by applying the Weil-Barner explicit formula to the test function
\begin{equation}\label{e:fabh-def}
 f_{a,b,h}(t) = \frac{1}{2\pi} \frac{e^{-ait} - e^{-bit}}{it} \frac{1}{\cosh(\tfrac h2 t)}
\end{equation}
which belongs to the Barner class for $h> \sigma_1 - \frac{\sigma_0}{2}$.

For $\abs{\Imn(z)} < \frac h2$ its Fourier transform is given by
\begin{equation}\label{e:Ffabh-def}
 \hat f_{a,b,h}(z) =\frac{2}{\pi}\Bigl[\arctan\bigl(\exp\bigl[\tfrac{\pi}{h}(z-a)\bigr]\bigr) - \arctan\bigl(\exp\bigl[\tfrac{\pi}{h}(z-b)\bigr]\bigr)\Bigr],
\end{equation}
where $\arctan(z)$ is holomorphically extended to  $\{z\in\C\mid  iz\notin [1,\infty)\cup (-\infty,-1]\}$. The function $\hat f_{a,b,h}$ is the convolution of $\chi_{[a,b]}$ and  $\frac{1}{h\cosh(\frac{\pi}{h} z)}$ from which it inherits the property of having positive real part in this strip. So if the zeros in the list are used to approximate $w_s(\hat f_{a,b,h})$, the Weil-Barner explicit formula thus gives an upper bound for contribution of the zeros not contained in the list.

The additional parameter $h$ controls a trade-off between the zeros being sacrificed and the number of summands which are needed to  approximate $w_f(f_{a,b,h})$ within sufficient accuracy.

The general method is based on the following theorem.

\begin{theorem}\label{t:fundamental-inequality}
 Let $(\sigma_0/2 + i\gamma_j)_{j=1}^m$ be a list of non-trivial zeros of $L$.
Let $h> 2\sigma_0 - \sigma_1$ and assume that $b-a > 5 \frac{h}{\pi}$ holds. Then, if the inequality
\begin{equation}\label{e:fundamental-inequality} 
 w_f(f_{a,b,h}) + w_\infty(f_{a,b,h})
- \sum_{j=1}^m  \hat f_{a,b,h}(\gamma_j) 
+ \sum_{\rho\in\mathcal P(\Lambda)} \abs{n_\rho}  \hat f_{a,b,h}\bigl(\tfrac\rho i -\tfrac{\sigma_0}{2i}\bigr)
 \leq  0.49
\end{equation}
holds, the list contains all non-trivial zeros of $L(s)$ with imaginary part in $[a,b]$.

Conversely, under the assumption of the Riemann Hypothesis for $L(s)$, for every $\eps>0$ there exists a $B_{\eps}>0$ such that
\[
C_{\eps}(X) = \frac h\pi \log\log\bigl(e(Q+1)(\abs X +1)\bigr) + B_{\eps}
\]
satisfies the property that if the list contains all zeros with imaginary part in \linebreak $[a-C_{\eps}(a),b+C_{\eps}(b)]$, then the left hand side of \eqref{e:fundamental-inequality} is smaller than $\eps$.
\end{theorem}

The second part of the theorem also holds without the assumption of the Riemann Hypothesis, but the proof is much shorter this way.

First, we prove some bounds for the functions $f_{a,b,h}$ and $\hat f_{a,b,h}$. 

\begin{lemma}\label{l:test-function}
 Let $a<b$, $h>0$ and let  $\abs{\Imn(z)} < \frac h2$. Then the following inequalities hold:
\begin{align}
& \Ren \hat f_{a,b,h}(z) > 0\label{e:Ffabh-lower-global} \\
 &0 < 1 - \Ren \hat f_{a,b,h}(z) <  \frac 4\pi e^{-\tfrac{\pi}{h}\min\{\Ren(z) - a, b -\Ren(z)\}}	&\text{for $\Ren(z)\in [a,b]$}	\label{e:Ffabh-upper-lower-inner},\\
 &\Ren \hat f_{a,b,h}(z) < \frac 2\pi e^{-\tfrac{\pi}{h}\max\{\Ren(z)-b,a-\Ren(z)\}}	&\text{for $\Ren(z)\notin (a,b)$.}\label{e:Ffabh-upper-outer}
\end{align}
If we additionally assume $b-a>5\frac{h}{\pi}$, we have
\begin{equation}\label{e:Ffabh-lower-inner}
\Ren \hat f_{a,b,h}(z) > 0.49
\end{equation}
for $\Ren(z)\in[a,b]$.

Furthermore, we have
\begin{equation}\label{e:fabh-abs-upper}
\abs{f_{a,b,h}(t)} < \frac2\pi \frac{e^{-\frac{h}{2}\abs{t}}}{\abs{t}} 
\end{equation}
for $\abs{t}\geq 1$ and arbitrary $a,b\in\R$.
\end{lemma}

\begin{proof}
Since we have
\[
 2\arctan e^z = \int_{-\infty}^{\Ren(z)} \frac{dt}{\cosh(t+i\Imn(z))},
\]
the bound in \eqref{e:Ffabh-lower-global} follows from $\frac{1}{\cosh(z)}$ having positive real part in $\abs{\Imn(z)}<\frac \pi 2$. Furthermore, since for $\Ren(z)< 0$ and $\abs{\Imn(z)}< \frac \pi2$ the functions $\Imn \frac{1}{\cosh(z)}$ and $\Imn z$ have the same sign, we get
\begin{equation}\label{e:arctan-exp-bound}
 0 \leq \Ren \arctan e^z \leq \arctan e^{\Ren(z)} = \sum_{n=0}^\infty \frac{(-1)^n}{2n+1} e^{(2n+1)\Ren(z)} \leq e^{\Ren(z)}
\end{equation}
for such $z$.

Now let $\Ren(z)\in[a,b]$. Since the chosen branch for $\arctan(z)$ satisfies
\[
\arctan(z) + \arctan(1/z) = \frac \pi 2
\]
in $\Ren(z)>0$, we get
\[
 \Ren \hat f_{a,b,h}(z) = 1 - \frac{2}{\pi}\left(\Ren \arctan e^{\frac \pi h (a-z)} + \Ren \arctan e^{\frac \pi h (z-b)}\right).
\]
This together with \eqref{e:arctan-exp-bound} gives the bound in \eqref{e:Ffabh-upper-lower-inner}. For \eqref{e:Ffabh-lower-inner} we also use the fact that the function $x\mapsto \Ren \hat f_{a,b,h}(x+iy)$ has a global maximum at $x=\frac{a+b}2$ and is otherwise monotonic. Thus we have
\[
 \Ren \hat f_{a,b,h}(z) \geq  \Ren \hat f_{a,b,h}(a+i \Imn y) \geq \frac12 - \frac{2}{\pi} e^{-\frac\pi h (b-a)} > 0.49,
\]
since $b-a \geq 5\frac h\pi$.

The bound in \eqref{e:Ffabh-upper-outer} follows similarly, e.g., for $\Ren(z) \geq b$ we have
\begin{equation*}
 \Ren \hat f_{a,b,h}(z) = \frac2\pi\left(\Ren \arctan e^{\frac\pi h(b-z)} - \Ren \arctan e^{\frac\pi h(a-z)}\right) 
\leq \frac2\pi e^{\frac\pi h(b-\Ren z)}.
\end{equation*}

The remaining inequality in \eqref{e:fabh-abs-upper} follows directly from \eqref{e:fabh-def}.
\end{proof}

\begin{proof}[Proof of Theorem \ref{t:fundamental-inequality}]
Let $(\rho_j)_{j=1}^\infty$ be the list of zeros that are missing in $\mcL$. Then, if the inequality in \eqref{e:fundamental-inequality} is satisfied, we have 
\begin{equation*}
\Ren \hat f_{a,b,h} \bigl(\tfrac{\rho_l}i -\tfrac{\sigma_0}{2i}\bigr) \leq \sum_{j=1}^\infty \Ren \hat f_{a,b,h}\bigl(\tfrac{\rho_j}i -\tfrac{\sigma_0}{2i}\bigr) \leq 0.49
\end{equation*}
for every $l\in\N$, where the first inequality follows from \eqref{e:Ffabh-lower-global} and the second follows from the explicit formula. Consequently, there is no $\rho$ with $\Imn(\rho)\in[a,b]$ among the $\rho_j$,  since by \eqref{e:Ffabh-lower-inner} every such zero would contribute an amount $>0.49$.

It remains to prove the second part. Assuming the Riemann Hypothesis, we have $\rho = \frac{\sigma_0}{2}+i\gamma$ with $\gamma\in\R$ for all $\rho\in\mathcal N(\Lambda)$. We apply the Weil-Barner explicit formula to the Fourier transform pair
\[
 g_{r,X}(t) = \frac12 e^{-r\abs{t}+iXt},\qquad \hat g_{r,X}(\xi) = \frac{1}{(\xi-X)^2+r^2}
\]
(taking $r>\sigma_1-\sigma_0/2$), and obtain the identity
\begin{equation}\label{e:aux-sum}
 \sum_{\rho \in \mathcal N (\Lambda)}\frac{n_\rho}{(\gamma -X)^2+r^2} 
= \Ren\frac{\Lambda'}{\Lambda}\Bigl(\frac{\sigma_0}{2}+r+iX\Bigr) 
+ \sum_{\rho\in \mathcal P(\Lambda)}\frac{\abs{n_\rho}}{(i(\frac{\sigma_0}{2} -\rho)-X)^2+r^2},
\end{equation}
the right hand side of which is $O\bigl(\log[(Q+2)(\abs{X}+2)]\bigr)$ by Stirling's formula and the uniform boundedness of the sum over poles. We take $X=b+C_\eps(b)$. Then, in view of \eqref{e:Ffabh-upper-outer}, we have
\[
\hat f_{a,b,h}(\gamma) \leq 1.3 e^{-\frac\pi h (C_\eps(b) + \gamma - X)} \ll_{r,h} \frac{e^{-\frac\pi h C_\eps(b)}}{r^2+(X-\gamma)^2}.
\]
Since there occur only positive summands on the left hand side of \eqref{e:aux-sum}, this implies
\[
 \sum_{\substack{\rho\in\mathcal N(\Lambda)\\ \gamma > b + C_{\eps,\delta}(b)}}n_\rho \hat f_{a,b,h}(\gamma)
 \ll_{r,h} e^{-\frac\pi h C_\eps(b)} \log[(Q+2)(\abs{X}+2)] \ll_{r,h} e^{-\frac\pi h B_\eps},
\]
which is $<\eps/2$ for $B_\eps$ sufficiently large. The considerations for $\gamma < a-C_\eps(a)$ are exactly the same, so the assertion follows.
\end{proof}

The remaining part of this section will be devoted to the evaluation of $w_f(f_{a,b,h})$, $w_\infty(f_{a,b,h})$, and the sum over zeros.


\subsection{Evaluation of $w_f(f_{a,b,h})$}
The value $w_f(f_{a,b,h})$ is approximated by evaluating (a usually small) part of the sum over prime powers. We give a simple estimate for the remainder.

\begin{lemma}\label{l:w_f}
Let $M\in \N_{>0}$, and let $C$ be the constant in \textsc{(L\ref{L:Euler-product})}. Then we have
\begin{equation}\label{e:wf-remainder}
\abs{2 \sum_{p^m \leq M} \frac{\log p}{\sqrt{p^{\sigma_0 m}}}\Ren\bigl(c(p^m)f_{a,b,h}(m\log p)\bigr)   + w_f(f_{a,b,h})} 
\leq \frac{8C}{\pi} \frac{M^{\sigma_1-\frac{h+\sigma_0}{2}}}{\sigma_0+h-2\sigma_1}.
\end{equation}
\end{lemma}
\begin{proof}
 By \eqref{e:fabh-abs-upper}, the left hand side of \eqref{e:wf-remainder} is bounded by
\[
 \frac{4C}{\pi}\sum_{n=M+1}^\infty n^{\sigma_1-1-\frac{\sigma_0+h}{2}} \leq \frac{4C}{\pi} \int_M^\infty t^{\sigma_1-1-\frac{\sigma_0+h}{2}}\, dt
= \frac{8C}{\pi} \frac{M^{\sigma_1 - \frac{\sigma_0+h}{2}}}{\sigma_0+h-2\sigma_1}.
\]
\end{proof}


\subsection{Evaluation of $w_\infty(f_{a,b,h})$}
The term $w_\infty(f_{a,b,h})$ gives an approximation to the number of zeros with imaginary part in $[a,b]$, which is closely related to the imaginary part of a branch of $\log G$. To avoid ambiguity the following notation will be used.

\begin{definition}
Let $U\subset\C$ be open and convex and let $f\colon U\rightarrow \C$ be holomorphic and non-vanishing. Then, for any $w\in U$ we define
\[
l_{f,w}(z) = \int_{w}^{z} \frac{f'}{f}(\xi)\, d\xi.
\]
\end{definition}

We will need the obvious properties $l_{fg,w} = l_{f,w} + l_{g,w}$ and
\[
l_{f,w}(z) - l_{f,w}(z') = l_{f,w'}(z) - l_{f,w'}(z').
\]
In particular we will use
\begin{equation*}
l_{\Gamma,1}(z) = \int_{0}^\infty (z-1)\frac{e^{-t}}{t} + \frac{e^{-zt}-e^{-t}}{t(1-e^{-t})}\, dt
\end{equation*}
for $\Ren(z)>0$. This is the branch of $\log\Gamma(z)$ for which we have the Stirling formula
\[
l_{\Gamma,1}(z) = \Bigl(z-\frac 12\Bigr) \log(z) - z + \frac 12 \log(2\pi) + O(1/\abs{z}),
\]
where $\log(z)$ denotes the principle value logarithm.

\begin{lemma}\label{l:w_8}
We have
 \begin{multline}\label{e:w_inf}
  w_\infty(f_{a,b,h}) =
   \frac{1}{\pi} \Imn\Bigl[l_{G,\sigma_0}\bigl(\sigma_0/2+ib\bigr) - l_{G,\sigma_0}\bigl(\sigma_0/2+ia\bigr)\Bigr] \\
-\sum_{k=1}^{k_1} \frac{1}{\pi} \int_{0}^\infty \frac{e^{-(\frac{\lambda_k\sigma_0}{2} + u_k)t}}{1-e^{-t}}
\frac{\sin((\lambda_kb \! +\! v_k)t) - \sin((\lambda_ka\! + \!v_k)t)}{t}\Bigl(\frac{1}{\cosh(\frac{\lambda_k h}{2}t)}-1\Bigr)\, dt.
\end{multline}
\end{lemma}

\begin{proof}
 From \eqref{e:fabh-def} we get
\begin{multline*}
 w_f(f_{a,b,h}) = \frac{b-a}{\pi}\frac{G'_0}{G_0}(\sigma_0/2) \\
- \sum_{k=1}^{k_1} \frac 1\pi \int_{0}^\infty \frac{e^{-(\frac{\lambda_k\sigma_0}2 + u_k)t}}{1-e^{-t}}
\Bigl(\frac{\sin((\lambda_kb+v_k)t)-\sin((\lambda_k a+v_k)t)}{t\cosh(\frac{\lambda_kh}{2}t)} - \lambda_k(b-a)\Bigr)\, dt.
\end{multline*}
Using
\begin{equation*}
 \lambda_k\frac{b-a}{\pi}\frac{\Gamma'}{\Gamma}(\lambda_k\sigma_0/2 + u_k) =
 \lambda_k\frac{b-a}{\pi}\int_{0}^\infty \Bigl(\frac{e^{-t}}{t} -  \frac{e^{-(\frac{\lambda_k\sigma_0}2 + u_k)t}}{1-e^{-t}}\Bigr)\, dt
\end{equation*}
and
\begin{multline*}
 \Imn[l_{\Gamma,1}(\lambda_k(\tfrac{\sigma_0}{2} + ib) + \mu_k) - l_{\Gamma,1}(\lambda_k(\tfrac{\sigma_0}{2} + ia) + \mu_k)] \\
= \int_{0}^\infty \Bigl(\frac{e^{-(\frac{\lambda_k\sigma_0}2 + u_k)t}}{1-e^{-t}}
(\sin((\lambda_kb+v_k)t)-\sin((\lambda_ka + v_k)t) + \lambda_k(b-a)e^{-t}\Bigr)\, \frac{dt}{t},
\end{multline*}
we get
\begin{multline}\label{e:integral-to-gamma}
 \frac 1\pi \int_{0}^\infty \frac{e^{-(\frac{\lambda_k\sigma_0}2 + u_k)t}}{1-e^{-t}}
\Bigl(\frac{\sin((\lambda_kb+v_k)t)-\sin((\lambda_k a+v_k)t)}{t\cosh(\frac{\lambda_kh}{2}t)} - \lambda_k(b-a)\Bigr)\, dt \\
= \frac1\pi\Imn[l_{\Gamma,1}(\lambda_k(\tfrac{\sigma_0}{2} + ib) + \mu_k) - l_{\Gamma,1}(\lambda_k(\tfrac{\sigma_0}{2} + ia) + \mu_k)]
- \lambda_k\frac{b-a}{\pi}\frac{\Gamma'}{\Gamma}(\lambda_k\sigma_0/2 + u_k) \\
+  \frac{1}{\pi} \int_{0}^\infty \frac{e^{-(\frac{\lambda_k\sigma_0}{2} + u_k)t}}{1-e^{-t}}
\frac{\sin((\lambda_kb+v_k)t) - \sin((\lambda_ka + v_k)t)}{t}\Bigl(\frac{1}{\cosh(\frac{\lambda_k h}{2}t)}-1\Bigr)\, dt.
\end{multline}
If we take into account that
\[
 \frac{b-a}{\pi} \log Q = \frac{1}{\pi}\Imn\Bigl[l_{Q^\cdot,1}\Bigl(\frac{\sigma_0}2 + ib\Bigr) - l_{Q^\cdot,1}\Bigl(\frac{\sigma_0}2 + ia\Bigr)\Bigr], 
\]
the assertion follows by summing \eqref{e:integral-to-gamma} over $k$ and adjusting the base points.
\end{proof}

Next, we will estimate the integrals in \eqref{e:w_inf}. They turn out to be of small modulus when $\abs{\lambda_k a+ v_k}$ and $\abs{\lambda_k b + v_k}$ are sufficiently large.

\begin{lemma}\label{l:w8_rem}
 Let $R\in\R\setminus\{0\}$,
\[
 0 < B_k <\min\left\{2\pi, \frac{\pi}{\lambda_k h}\right\}
\]
and let
\[
 C_k = \frac{1}{B_k(1-\cos(B_k))}\left(1+\frac 1{\cos(\frac{\lambda_k h}{2}B_k)}\right).
\]
Then we have
\begin{equation}\label{e:osc-int}
\abs{\int_{0}^\infty \frac{e^{-(\frac{\lambda_k\sigma_0}{2} + u_k)t}}{1-e^{-t}}
\frac{\sin(Rt)}{t}\Bigl(\frac{1}{\cosh(\frac{\lambda_k h}{2}t)}-1\Bigr)\, dt}
\leq C_k\Bigl(\frac{1}{\abs{R}} + \frac{e^{-\abs{R}B_k}}{\frac{\lambda_k\sigma_0}{2} + u_k}\Bigr).
\end{equation}
\end{lemma}
\begin{proof}
 Let $A = \frac{\lambda_k\sigma_0}{2}+ u_k$ and let
\[
 g(z) = \frac{1}{z(1-e^{-z})}\left(\frac{1}{\cosh(\frac{\lambda_kh}{2}z)} - 1\right).
\]
Then $g(z)$ is holomorphic in $\abs{\Imn(z)}<\min\{2\pi, \frac{\pi}{\lambda_k h}\}$ and since
we have \[\abs{\cosh(z)}\geq \abs{\cos(\Imn(z))}\] and 
\[
 \abs{1-e^{-t}}^2= 1-2 \cos(\Imn(z))e^{-\Ren(z)} + e^{-2\Ren(z)} \geq (1-\cos(\Imn(z))e^{-\Ren(z)})^2,
\]
the Phragm\'en-Lindel\"of principle gives
\begin{equation}\label{e:g-bound}
 \abs{g(z)} \leq \frac{1}{B_k(1-\cos(B_k))} \left(\frac{1}{\cos(\frac{\lambda_k h}{2}B_k)}+1\right) = C_k
\end{equation}
for $\abs{\Imn(z)}\leq B_k$. Since we have
\begin{multline*}
 \int_{0}^\infty e^{(\pm iR-A)t} g(t)\, dt = \pm i\int_0^{B_k} e^{-(R+iA)t}g(\pm it) \, dt \\
+ \int_{0}^\infty e^{(\pm iR-A)(\pm iB_k+t)}g(t\pm iB_k)\, dt,
\end{multline*}
where by \eqref{e:g-bound} the first integral is bounded by
\[
 C_k \int_{0}^\infty e^{-Rt} \, dt = \frac{C_k}{R}
\]
and the second integral is bounded by
\[
 C_k e^{-B_k R} \int_{0}^\infty e^{-At} \, dt = \frac{C_k e^{-B_k R}}{A},
\]
the assertion follows.
\end{proof}

If either $a$ or $b$ is close but not equal to $-v_k/\lambda_k$, these bounds are insufficient in order to apply the method. In such (rare) cases one could use numerical integration to evaluate the critical integrals in \eqref{e:w_inf}.

\subsection{Evaluation of the sum over zeros}
It is not actually necessary to evaluate $\hat f_{a,b,h}(\gamma_j)$ in the form \eqref{e:Ffabh-def} for all $j$ in \eqref{e:fundamental-inequality} (which would make this method inefficient compared to the Turing method). If we take e.g. $R=\frac h\pi (\log m + 5)$, where $m$ is the number of zeros in the list, and use the approximation $\hat f_{a,b,h}(\gamma_j)\approx 1$ for $\gamma_j\in [a+R,b-R]$, this results in a total error $< \frac{1}{100}$. Therefore, as in the case of the Turing method, only the zeros with imaginary part in a neighbourhood of $a$ and $b$ are needed within an accuracy of $O(1/\log\abs a)$ resp. $O(1/\log \abs b)$, which is well in the range of the average spacing between consecutive zeros.

\section{Examples}
We further specify this method to some well-known families of L-functions, for which we also give an explicit converse statement.

\subsection{The Riemann zeta function}\label{ss:rieman-zeta}
The case of the Riemann zeta function has already been carried out in \cite{BJ10} and \cite{BFJK13}. Unfortunately, in \cite{BFJK13} there is a mistake concerning the sign of the terms $\frac{R}{2\pi}\log \pi$ and $\frac{b-a}{2\pi}\log \pi$ in equations (4.21) and (4.25).

We restate the results in a more general form and give an improved estimate for the length of the cut-off interval.

\begin{theorem}\label{t:zetaR}
Let $R\geq 15$, let $h\in(1,\pi]$, and let $\alpha = \frac{h-1}{2}$. Let $\mcL= (\frac12 + it_n)_{n=1}^N$ be a list of the zeros of the Riemann zeta function. Then, if the inequality
\begin{multline}\label{e:zetaR}
-\frac{1}{\pi}\sum_{p^m\leq (30/\alpha)^{1/\alpha}} \frac{\sin(Rm\log p)}{mp^{\frac m2}\cosh(\frac h2 m\log p)} \\
+ \frac 1\pi \Imn\Bigl[l_{\Gamma,1}\Bigl(\frac 14 + i\frac R2\Bigr)\Bigr] 
- \frac R{2\pi}\log \pi + \frac{1.6}{R}  \\
- \sum_{n=1}^N \hat f_{0,R,h}(t_n) \leq - 0.56
\end{multline}
holds, where $\hat f_{0,R,h}$ is the function defined in \eqref{e:Ffabh-def}, the list contains all zeros with imaginary part in $(0,R]$.

Conversely, if $R\geq 10^6$ holds in addition to the previous assumptions, and if $\mcL$ contains all zeros with imaginary part in $\bigl(0,R+\frac h\pi(\log\log R + 0.4)\bigr]$, then the inequality in \eqref{e:zetaR} holds.
\end{theorem}

\begin{proof}
Let $\alpha$ be as in the theorem. Then, by Lemma \ref{l:w_f} the sum on the first line of \eqref{e:zetaR} differs at most by
\[
\frac{4}{\pi\alpha} M^{-\alpha} = \frac{2}{15\pi} < 0.043
\]
from $w_f(f_{0,R,h})$.

In Lemma \ref{l:w8_rem} we take $B_1 = \frac{4.9}{h}$. Then $C_1 = C_1(h)$ takes its maximum on $[1,\pi]$ at $h=\pi$, which is $<2.6$. Therefore, the oscillatory integral in \eqref{e:w_inf} is bounded by
\[
\frac{2.6}{\pi}\Bigl(\frac{2}{R} + 4e^{-\frac{4.9 R}{2 h}}\Bigr) < \frac{1.6}{R}.
\]
Consequently, by Lemma \ref{l:w_8}, $w_{\infty}(f_{0,R,h})$ does not exceed the value on the second line of \eqref{e:zetaR}.

Finally, from  \eqref{e:arctan-exp-bound} we get
\[
2\Ren \hat f_{0,R,h} \geq 1 - \frac{4}{\pi} e^{-\frac \pi h R} > 1 - 10^{-6}, 
\]
so the first assertion, concerning the completeness of the list of zeros, follows from Theorem \ref{t:fundamental-inequality}.

For the second part of the theorem, we will also need the following lemma, whose proof we postpone to the end of this section.

\begin{lemma}\label{l:zeta-zsum}
 Let $a\in \{0\}\cup (14,\infty)$, let $b>\max\{a,14\}$ and let $h\in(1,\pi]$. Then, for $T_b > b$ we have
\begin{equation}\label{e:zsum-upper}
 \sum_{\Imn(\rho)>T_b} n_\rho \hat f_{a,b,h}\Bigl(\frac\rho i - \frac 1{2i}\Bigr) \leq e^{\frac\pi h(b-T_b)}\Bigl[(0.143+0.033h)\log T_b
 + 0.354\log\log T_b + 3.3\Bigr].
\end{equation}
Also, we have
\begin{equation}\label{e:zsum-lower1}
  \sum_{\Imn(\rho)< 0} n_\rho \hat f_{a,b,h}\Bigl(\frac\rho i - \frac 1{2i}\Bigr) \leq \frac{e^{-\frac\pi h a}}{10000},
\end{equation}
and for $14< T_a < a$ we have
\begin{equation}\label{e:zsum-lower2}
 \sum_{0<\Imn(\rho)<T_a} n_\rho \hat f_{a,b,h}\Bigl(\frac\rho i - \frac 1{2i}\Bigr) \leq e^{\frac\pi h(T_a-a)}
\Bigl[(0.143+0.033h)\log a  + 0.354\log\log a + 3.3\Bigr].
\end{equation}
\end{lemma}

Now, assuming $R\geq 10^6$, let $(\rho_j)_{j=1}^\infty$ be the list of zeros not contained in $\mcL$. Then, by \eqref{e:zsum-upper} and \eqref{e:zsum-lower1} we have
\begin{multline}
\sum_{j=1}^{\infty} \hat f_{0,R,h} \Bigl(\frac\rho i - \frac 1{2i}\Bigr) \leq \frac{1}{10000} + \frac{e^{-0.4}}{\log(R)}\Bigl[(0.143 + 0.033 h)\log (R) \\
 + 0.354\log\log R + 3.31\Bigr] \leq 0.38,
\end{multline}
where we also used $\log\log(R+C) \leq \log\log R + \frac{C}{R\log(R)}$ and $\log(R+C) \leq \log(R) + C/R$.
Hence, the left hand side of \eqref{e:zetaR} is smaller than
\[
0.043 + \frac{1.6}{R} - 1 + 10^{-6} + 0.38 < -0.57.
\]

\end{proof}

For general subintervals $[a,b]$ of the positive real numbers, we get the following theorem.

\begin{theorem}\label{t:zetaab}
Let $h\in(1,\pi]$, let $15<a<b-5\frac{h}{\pi}$ and let $\alpha = \frac{h-1}{2}$. Let \linebreak $\mcL = (\frac12 + it_n)_{n=1}^N$ be a list of the zeros of the Riemann zeta function. Then, if the inequality
\begin{multline}\label{e:zetaab}
-\frac{2}{\pi}\sum_{p^m \leq (\frac{30}\alpha)^{1/\alpha}} \frac{\sin(\frac{b-a}{2}m\log p)}{mp^{m/2}}\frac{\cos(\frac{a+b}{2}m\log p)}{\cosh(\frac\pi2 m\log p)} \\
+\frac1\pi\Imn\Bigl[l_{\Gamma,1}\Bigl(\frac 14 + i\frac b2\Bigr)- l_{\Gamma,1}\Bigl(\frac 14 + i\frac a2\Bigr)\Bigr] 
- \frac{b-a}{2\pi}\log \pi + \frac{3.2}{a} \\
 - \sum_{n=1}^N \hat f_{a,b,h}(t_n) \leq 0.44
\end{multline}
holds, where $\hat f_{a,b,h}$ is the function defined in \eqref{e:Ffabh-def}, the list $\mcL$ contains all zeros of the zeta function with imaginary part in $[a,b]$.

Conversely, if $a\geq 10^6$ in addition to the previous assumptions on $a$ and $b$, and if $\mcL$ contains all zeros with imaginary part in $[a-C(a)), b+C(b)]$, where
\[
C(T) = \frac{h}{\pi}(\log\log T + 1.1),
\]
then \eqref{e:zetaab} holds.
\end{theorem}

\begin{proof}
 As in the proof of Theorem \ref{t:zetaR} we see that the expression on the first line of \eqref{e:zetaab} exceeds $w_f(f_{a,b,h})$ by at most $0.043$, and the expression on the second line is larger than $w_{\infty}(f_{a,b,h})$. In a similar way, we see that the pole contribution is now bounded by $-10^{-6}$, so the first assertion follows from Theorem \eqref{t:fundamental-inequality}.

Now let $(\rho_j)_{j=1}^\infty$ denote again the list of zeros not contained in $\mcL$. Then, by Lemma \ref{l:zeta-zsum} we have
\begin{multline*}
\sum_{j=1}^{\infty} \hat f_{a,b,h} \Bigl(\frac\rho i - \frac 1{2i}\Bigr) \leq \frac{e^{-\frac \pi h 10^6}}{10000} + 2 e^{-1.1}\Bigl[ 0.143 + 0.033 h \\
 + 0.354\frac{\log\log a }{\log(a)} + \frac{3.31}{\log a}\Bigr] \leq 0.37,
\end{multline*}
and since we have
\[
 0.043 + \frac{3.2}{10^6} + 10^{-6} + 0.37 < 0.42,
\]
the second assertion follows.
\end{proof}


It remains to prove the lemma.

\begin{proof}[Proof of Lemma \ref{l:zeta-zsum}]
 Let
\[
 g(T) = \frac{T}{2\pi}\log\frac{T}{2\pi e} + \frac{7}{8}
\]
and let $r(T) = N(T) - g(T)$, where $N$ is the function that counts the zeros (according to their multiplicity) of the Riemann zeta function with imaginary part in $(0,T]$. Then we have
\begin{equation}\label{e:r-bound}
 \abs{r(T)} \leq r_1(T) = 0.112 \log T + 0.278 \log\log T + 2.584
\end{equation}
for $T\geq e$ \cite{Trudgian14}. From \eqref{e:Ffabh-upper-outer} we get
\begin{align*}
\sum_{\Imn(\rho)>T_b} n_\rho \hat f_{a,b,h}\Bigl(\frac\rho i - \frac 1{2i}\Bigr) 
  &\leq \frac 2\pi \int_{T_b}^\infty e^{\frac\pi h (b-t)} \, dN(t) \\
  &= -\frac{2}{\pi} e^{\frac\pi h (b-T_b)}N(T_b) + \frac 2h \int_{T_b}^\infty e^{\frac\pi h (b-t)} (g(t) + r(t))\, dt \\
  &\leq \frac{2}{\pi} e^{\frac\pi h (b-T_b)} r_1(T_b) + \frac 2\pi \int_{T_b}^\infty e^{\frac\pi h (b-t)} g'(t)\, dt \\
&\quad\quad +  \frac 2h \int_{T_b}^\infty e^{\frac\pi h (b-t)} r_1(t)\, dt.
\end{align*}
Here we use $g'(t) = \frac1{2\pi}\log\frac{t}{2\pi}$, the bound in \eqref{e:r-bound} and the obvious inequalities
\begin{equation*}
 \int_T^\infty e^{-\frac \pi h t} \log (at) \, dt \leq \frac h\pi e^{-\frac\pi h T} \Bigl(\log (aT) + \frac{h}{\pi T}\Bigr)
\end{equation*}
and
\begin{equation*}
  \int_T^\infty e^{-\frac \pi h t} \log\log t \, dt \leq \frac h\pi e^{-\frac\pi h T} \Bigl(\log\log T + \frac{h}{\pi T\log T}\Bigr)
\end{equation*}
which give the inequality in \eqref{e:zsum-upper}.

For the inequality in \eqref{e:zsum-lower1} we use the well-known fact, that we have $N(T) = 0$ for $0\leq T\leq 14$ and that the zeros of the Riemann zeta function are symmetric about the real axis, which gives
\begin{equation*}
 \sum_{\Imn(\rho)< 0} n_\rho \hat f_{a,b,h}\Bigl(\frac\rho i - \frac 1{2i}\Bigr) 
\leq \frac 2 h \int_{14}^\infty e^{-\frac \pi h(t+a)} N(t)\, dt.
\end{equation*}
Here we use the bound
\[
 N(T) \leq \frac{T}{2\pi} \log T,
\]
which for $t\geq 30$ follows from \eqref{e:r-bound} and for $14\leq t\leq 30$ from well-known numeric results (e.g. \cite{Brent79}), and the fact that $1<h\leq \pi$. A simple calculation then confirms the bound in \eqref{e:zsum-lower1}.

In order to prove \eqref{e:zsum-lower2} we proceed in a similar way as in the proof of \eqref{e:zsum-upper}. We have
\begin{align*}
\sum_{0<\Imn(\rho)<T_a} n_\rho \hat f_{a,b,h}\Bigl(\frac\rho i - \frac 1{2i}\Bigr) 
  &\leq \frac 2\pi \int_e^{T_a} e^{\frac\pi h(t-a)} \, dN(t) \\
  &\leq \Bigl[\frac{2}{\pi}e^{\frac\pi h(t-a)} r(t)\Bigr]_e^{T_a} + \frac{2}{\pi}\int_e^{T_a} e^{\frac\pi h(t-a)} g'(t)\, dt \\
 &\quad\quad + \frac{2}{h}\int_e^{T_a} e^{\frac\pi h(t-a)} r_1(t)\, dt.
\end{align*}
Here we have
\[
 \Bigl[\frac{2}{\pi}e^{\frac\pi h(t-a)} r(t)\Bigr]_e^{T_a} 
\leq \frac{2}{\pi}e^{\frac\pi h(T_a-a)}\Bigl(r_1(T_a) + e^{\frac\pi h(e-T_a)} r_1(e)\Bigr)
\]
and the integrals can be estimated using the the monotony of $g'(t)$ and $r_1(t)$.
\end{proof}

\begin{remark}
We would like to compare these results to the Turing method, which is based on the following identities:
\begin{align}\label{e:Turing-identities}
 N(T) &= - \sum_{T<\Imn(\rho) < T+y}n_\rho\frac{T+y-\Imn(\rho)}{y} + \frac 1y\int_{T}^{T+y}\theta(t)\, dt + \frac 1y\int_T^{T+y} S(t)\, dt \\
      &= \sum_{T-y<\Imn(\rho)<T} n_\rho \frac{\Imn(\rho) + y - T}{y} + \frac 1y\int_{T-y}^{T}\theta(t)\, dt + \frac 1y\int_{T-y}^{T} S(t)\, dt,\notag
\end{align}
where 
\[
\theta(t) = \frac{1}{\pi}\Imn l_{\Gamma,1} \Bigl(\frac 14 + i\frac t2\Bigr) 
- \frac t{2\pi} \log \pi + 1
\]
and $S(T) = N(T) - \theta(T)$. From this one derives upper and lower bounds for $N(T)$ using approximations to the zeros with imaginary part in $(T,T+y)$ resp. $(T,T-y)$, where the integral is bounded by

\begin{equation}\label{e:S-bound}
 \abs{\int_{t_1}^{t_2} S(t)\, dt} \leq 2.067 + 0.059 \log t_2
\end{equation}
for $t_2\geq t_1\geq 168\pi$ \cite{Trudgian2011B}. So for $y$ sufficiently large, the number of zeros is determined exactly by \eqref{e:Turing-identities} if all zeros with imaginary part in $(T-y,T)$, resp. $(T,T+y)$ are known.

For a comparison we choose $h=2.5$ in Theorem \ref{t:zetaR} and Theorem \ref{t:zetaab}, evaluating the sum over all prime powers $\leq 140$. If we take $R\geq 10^6$ in Theorem \ref{t:zetaR} (resp. $a\geq 10^{6}$ in Theorem \ref{t:zetaab}) it follows from \eqref{e:S-bound}, that for a safe application of these methods the minimal length of the cut-off interval has to be $2.5$ times (resp. $2$ times) as large for the Turing method. Nevertheless, all these methods usually succeed with smaller cut-off intervals. E.g. some numerical tests with $a,b,R$ and $T$ close to $10^{10}$ suggest, that for the Turing method  a cut-off interval of length $4.5$, for the method in Theorem \ref{t:zetaab} one of length of $1.8$ and for the method in Theorem \ref{t:zetaR} one of length of 1 is sufficient.

With the methods as stated above, the Turing method has a slight advantage with respect to the minimal required precision of the zeros in the list. For the method in Theorem \ref{t:zetaab} this precision is about four times as high as for the Turing method and for the method in Theorem \ref{t:zetaR} it is about twice as high. This should not impose practical restrictions, especially when fast evaluation techniques on grids are applied, but if one would like to apply the new methods with very coarse approximations to the zeros, the minimal required precision could be cut in half at the expense of doubling the minimal length of the cut-off interval.

There also exists a variant of Turing's method due to Lehmer, for which it suffices to find a certain number of successive gram blocks satisfying the Rosser rule \cite{Lehman70}. This has been left out of the comparison, because very little can be said about the length of the cut-off interval (the Rosser rule eventually fails a positive proportion of the time \cite{Trudgian11}) or the accuracy which is required to separate the zeros.
\end{remark}

\subsection{Hecke L-series}
Let $K$ be an algebraic number field of degree $N$ over $\Q$ with absolute discriminant $d_K$. We denote the number of real archimedian places by $r_1$ and the number of complex places by $r_2$. By $\chi$ we denote a Hecke gr\"o\ss encharacter and by $\mathfrak f$ its conductor. For $\Ren(s)>1$ the Hecke L-function for the character $\chi$ is given by
\begin{equation*}
L(s,\chi) = \sum_{\mfa} \frac{\chi(\mfa)}{N\mfa ^s} = \prod_\mfp \left(1 - \frac{\chi(\mfp)} {N\mfp^s}\right)^{-1},
\end{equation*}
where the sum is taken over all ideals in $\mcO_K$ and the product is taken over all prime ideals. We take
\begin{equation*}
G(s,\chi) = \left(\frac{\abs{d_k}N\mff}{4^{r_2}\pi^N}\right)^{s/2} \prod_{j=1}^{r_1} \Gamma\left(\frac{s + i\varphi_j + n_j}{2}\right)\prod_{j=r_1+1}^{r_1+r_2} \Gamma\left(s+i\varphi_j + \frac{\abs{n_j}}{2}\right),
\end{equation*}
where the numbers $\varphi_j\in\R$ and $n_j\in \Z$ uniquely determine $\chi$ at the archimedian places. Then the complete L-function \begin{equation*}
\Lambda(s,\chi) = G(s,\chi) L(s,\chi)
\end{equation*}
satisfies the functional equation
\begin{equation*}
\Lambda(s,\chi) = W(\chi) \Lambda(1-s,\overline \chi),
\end{equation*}
where $W(\chi)$ is the root number of $\chi$, a complex number of modulus 1.

For $L(s,\chi)$ we have the following result.

\begin{theorem}\label{t:hecke-l-series}
Let $\mathcal L = (\frac12 + it_j)_{j=1}^n$  be a list of non-trivial zeros of $L(s,\chi)$, and let $b-a>5$ such that for $z\in \{a,b\}$ and $j=1,\dots , r_1+r_2$ we have either $z+\varphi_j=0$ or $\abs{z+\varphi_j}> 20 N$. Let $\eps_0=1$ if $\chi$ is the principal character and else let $\eps_0=0$. Furthermore, let
\[
\mcE_0(t) =
\begin{cases}
0	&	t=0,\\
\frac{1}{\abs t} & \text{else},
\end{cases}
\]
and let
\[
\mcE(a,b) = 1.65 \sum_{j=1}^{r_1}  \left[\mcE_0(a+\varphi_j) + \mcE_0(b+\varphi_j)\right] + 5.57 \sum_{j=r_1+1}^{r_1+ r_2}  \left[\mcE_0(a+\varphi_j) + \mcE_0(b+\varphi_j)\right].
\]
Then, if the inequality
\begin{multline}\label{e:hecke}
-2\Ren\left(\sum_{N\mfp^m\leq 20N} \frac{\log N\mfp}{N\mfp^{m/2}}\chi(\mfp^m)f_{a,b,\pi}(m\log N\mfp) \right) \\
+ \frac1\pi \Imn\left[l_{G,1}\left(\frac12+ib,\chi\right) - l_{G,1}\left(\frac12+ia, \chi\right)\right] + \mcE(a,b) \\
- \sum_{j=1}^n \hat f_{a,b,\pi}(t_n) +2 \eps_0 \Ren(\hat f_{a,b,\pi}(i/2)) \leq 0.44
\end{multline}
holds, where $f_{a,b,\pi}$ and $\hat f_{a,b,\pi}$ are the functions defined in \eqref{e:fabh-def} and \eqref{e:Ffabh-def}, $\mcL$ contains all zeros with imaginary part in $[a,b]$.

Conversely, let
\[
Q'= \left(\frac{\abs{d_K}N\mff}{4^{r_2}\pi^N}\right)^{1/2} +e,\quad\quad A = \max_{j}\{\abs{\varphi_j}\} +  \max_j\{\abs{n_j}/2\},
\]
and let
\[
C(X) = \log\log(\abs{X}+A) + \log\log Q' + \log(N) + 3.
\]
We also assume that for $z\in \{a,b\}$ we have
\begin{equation}\label{e:hecke-converse-lower}
\abs{z} \geq 10\log\log(A+10) + 10\log\log(Q'+10) + 10\log(N) + 50
\end{equation}
and $\abs{z+\varphi_j}> 100 N$ whenever $z+\varphi_j\neq 0$. Then, under the assumption of the Riemann Hypothesis for $L(s,\chi)$, if $\mcL$ contains all zeros with imaginary part in $[a-C(a),b+C(b)]$, the inequality in \eqref{e:hecke} holds.
\end{theorem}

This result can be simplified a little for larger values of $a$ and $b$. For example, if we assumed $\abs{z+\varphi_j}> 2000 N$, the term $\mcE(a,b)$ in \eqref{e:hecke} could be omitted.

\begin{proof}
Since $p$ splits in at most $N$ prime ideals in $\mcO_K$, the coefficients $c(p^m)$ satisfy the bound
\begin{equation}\label{e:hecke-cp-bound}
\abs{c(p^m)} = \abs{\sum_{N\mfp^k = p^m} \frac{\chi(\mfp^k)}{k}} \leq N.
\end{equation}
Therefore, by Lemma \ref{l:w_f}, the first line of \eqref{e:hecke} differs at most by
\[
\frac{8N}{\pi}\frac{(20 N)^{\frac{1-\pi}2}}{\pi - 1} \leq 0.05
\]
from $w_f(f_{a,b,\pi})$.

Next, we investigate $w_\infty(f_{a,b,\pi})$. For the real places we take $B_k=1.6$ in Lemma \ref{l:w8_rem}. Then we have $C_k < 2.58$ and for $R\in\{\frac{a+\varphi_k}2,\frac{b+\varphi_k}2\}$ the left hand side of \eqref{e:osc-int} is
bounded by
\[
2.58 \left(\frac{1}{\abs{R}} + \frac{e^{-1.6 \abs{R}}}{\frac14 + \frac{n_k}{2}}\right) \leq \frac{2.59}{\abs{R}},
\]
where we used $\abs{R}\geq 10$. For the complex places we choose $B_k = 0.8$. Then we have $C_k < 17.46$ and the left hand side of \eqref{e:osc-int} is bounded by $17.47/\abs{R}$. Since $2\cdot 2.58/\pi \leq 1.65$ and $17.47/\pi\leq 5.56$, it follows from Lemma \ref{l:w_8} that $w_\infty(f_{a,b,\pi})$ is bounded by the expression on the second line of \eqref{e:hecke}. Thus, the statement concerning the completeness of the list $\mcL$ follows from Theorem \ref{t:fundamental-inequality}.

For the converse statement, we assume that all non-trivial zeros are given in the form $\rho=\frac12 + i\gamma$. We then have to show that
\begin{equation}\label{e:hecke-converse}
\sum_{\gamma\notin[a-C(a),b+C(b)]} n_\rho \hat f_{a,b,\pi}(\gamma)  + 0.05 + 2\mcE(a,b) \leq 0.44
\end{equation}
holds. Under the conditions imposed on $a$ and $b$ we have $2 \mcE(a,b)\leq 0.06$, and it is therefore sufficient to show that the sum over zeros is $<0.33$.

We treat the upper part of the sum, where $\gamma > X= b+C(b)$, first, following the idea in the proof of Theorem \ref{t:fundamental-inequality}. Using the bound in \eqref{e:Ffabh-upper-outer}, we see that
\begin{equation*}
\hat f_{a,b,\pi}(\gamma) \leq \frac 2\pi e^{-C(b) + X-\gamma} \leq \frac 2\pi \frac{e^{-C(b)}}{1+(X-\gamma)^2}
\end{equation*}
holds for such $\gamma$, and consequently we get
\begin{equation}\label{e:hecke-zsum-upper}
\sum_{\gamma > X} n_\rho \hat f_{a,b,\pi}(\gamma) \leq \frac 2\pi e^{-C(b)} \left(\Ren\frac{\Lambda'}{\Lambda}\left(\frac32 + X,\chi\right) + \eps_0\right)
\end{equation}
from \eqref{e:aux-sum}. In order to estimate the contribution of the finite primes to the logarithmic derivative of $\Lambda(s,\chi)$ we use the bound (4.26) in \cite{BFJK13}, which together with \eqref{e:hecke-cp-bound} gives
\begin{equation}\label{e:hecke-finite-primes}
\abs{\frac{L'}{L}(3/2)} \leq -N\frac{\zeta'}\zeta(3/2) \leq 1.51 N.
\end{equation}

For the contribution of the infinite primes to the logarithmic derivative of $\Lambda(s,\chi)$ we will also need a version of the Stirling formula for the digamma function.

\begin{lemma}
Let $z\in\C$ have positive real part. Then we have
\begin{equation}\label{e:stirling}
\frac{\Gamma'}{\Gamma}(z) = \log(z) - \frac{1}{2z} + \Theta\left(\frac{3}{2\abs{z}^2}\right).
\end{equation}
\end{lemma}
\begin{proof}
By the Stirling formula for $l_{\Gamma,1}$ we have
\begin{equation*}
\frac{\Gamma'}{\Gamma}(z) = \log(z) - \frac{1}{2z} + \mu'(z),
\end{equation*}
where $\mu(z)$ is holomorphic in $\C\setminus(\infty,0]$ and satisfies $\abs{\mu(re^{i\phi})}\leq \frac{1}{12 r\cos(\phi/2)^2}$ \cite[2.4.2]{Remmert1990}. From this we get the claimed bound for $\mu'(z)$ using the Cauchy formula, integrating along the circle $\abs{\xi-z}=\frac {\abs{z}}3$, where we have $\abs{\xi}\geq \frac{2}{3}\abs{z}$ and $\cos(\phi/2)^2\geq \frac 1 4$.
\end{proof}

We treat the $\Gamma$-factors at the real places first. We claim the bound
\begin{equation}\label{e:gamma-real}
\frac12 \Ren \frac{\Gamma'}{\Gamma}\left(\frac34 + \frac{n_j}2 + i\frac{X+\varphi_j}{2}\right) \leq \log(\abs b + A)
\end{equation}
to hold for $j\leq r_1$. If $\abs{\frac34 + \frac{n_j}2 + i\frac{X+\varphi_j}{2}} \leq 10$, the Stirling formula \eqref{e:stirling} implies that the left hand side of \eqref{e:gamma-real} is bounded by $\frac 12(\log(10) + \frac 32 (\frac 43)^2) < 2.5$ and the right hand side is larger than $3.9$, so the bound holds. Otherwise, the $\Theta$-term in $\eqref{e:stirling}$ is $<0.02$ and by \eqref{e:hecke-converse-lower} we have $C(b)< \abs{b}/2$, so we get
\begin{align*}
\log\abs{\frac34 + \frac{n_j}2 + i\frac{X+\varphi_j}{2}} &\leq \log \frac{3/2 + n_j + C(b) + \abs{b} + A}{2} \\
&\leq \log\frac{\abs b + A}{2} + \frac{3/2+n_j+C(b)}{\abs{b}} \leq \log(\abs b + A)-0.15,
\end{align*}

An analogous calculation, considering first the case $\abs{\frac32 + \frac{\abs{n_j}}2 + i(X+\varphi_j)}<10$, shows that we have
\begin{equation}\label{e:gamma-complex}
 \Ren \frac{\Gamma'}{\Gamma}\left(\frac32 + \frac{\abs{n_j}}2 + i(X+\varphi_j)\right) \leq \log(\abs{b}+A) + 0.57,
\end{equation}
at the complex places.

So combining \eqref{e:hecke-finite-primes}, \eqref{e:gamma-real} and \eqref{e:gamma-complex}, we see that the right hand side of \eqref{e:hecke-zsum-upper} is bounded by
\begin{multline*}
\frac{2}{\pi}e^{-C(b)} \left(N(1.51 + \log(\abs{b}+A) + 0.57) + \log(Q') + 1\right) \\
\leq \frac{2}{\pi e^3}\left(\frac{1.51 + 0.57 + 1 + 1}{\log(50)} + 1\right) \leq 0.07.
\end{multline*}
Since the conditions on $a$ and $b$ are symmetric, the same bound holds for the lower part of the sum in \eqref{e:hecke-converse} and thus the assertion follows.
\end{proof}

\subsection{L-series of elliptic curves over $\Q$}

Let $E$ be an elliptic curve defined over $\Q$ and let $N$ be its conductor. Then the L-series attached to $E$ is defined by the Euler product
\[
L(E,s) = \prod_{p\mid N} \bigl(1-\epsilon(p) p^{-s}\bigr)^{-1} \prod_{p\nmid N}\bigl(1-a_p p^{-s} + p^{1-2s}\bigr)^{-1}, 
\]
where $\epsilon(p)$ is $1$, $-1$ or $0$ according to whether $E$ has split multiplicative, non-split multiplicative or additive reduction at $p$ and, $a_p$ is defined by
\[
\#E(\mathbb F_p) = p + 1- a_p .
\]
By the modularity theorem, the complete L-function
\[
\Lambda(E,s) = \left(\frac{\sqrt N}{2\pi}\right)^s\Gamma(s)L(E,s)
\]
satisfies the functional equation
\[
\Lambda(E,s) = \pm \Lambda(E,2-s)
\]
of a weight $2$ modular form. For $L(E,s)$ we get the following result.

\begin{theorem}
 Let $E$ be an elliptic curve over $\Q$, and let $N$ be its conductor. Let $\mathcal L=(1 + i\gamma_j)_{j=1}^m$ be a list of non-trivial zeros of $L(E,s)$ and let $a,b\in\R$ such that we have $b-a>5$ and such that for $z\in\{a,b\}$ we have either $\abs{z}>15$ or $z=0$. 
 We define
 \[
 \mathcal E(t) =
 \begin{cases}
0	&	t=0 \\
\frac{5.57}{\abs{t}}	&\text{else}. 
 \end{cases}
 \]
Then, if the inequality
\begin{multline}\label{e:elliptic}
 -\frac{1}{\pi}\sum_{p^m< 30} \frac{c(p^m)}{p^m} \frac{\sin(bm\log p)-\sin(am\log p)}{m\cosh(\frac\pi 2m\log p)}\\
 + \frac{b-a}{2\pi} \log \frac N{4\pi^2}
+ \frac{1}{\pi}\Imn\Bigl[l_{\Gamma,1}(1+ ib) - l_{\Gamma,1}(1+ia)\Bigr] 
\\
+ \mathcal E(a) + \mathcal E(b) 
 - \sum_{j=1}^m \hat f_{a,b,\pi}(\gamma_j) \leq 0.42
\end{multline}
holds, where $\hat f_{a,b,\pi}$ is the function defined in \eqref{e:Ffabh-def}, $\mathcal L$ contains all non-trivial zeros of $L(E,s)$ with imaginary part in $[a,b]$.

Conversely, if the Riemann Hypothesis for $L(E,s)$ holds, if $a$ and $b$ additionally satisfy the condition
\begin{equation}\label{e:elliptic-converse-lower}
 \min\{\abs a,\abs b\} \geq \max\{70, 3\log\log N\},
\end{equation}
and if $\mathcal L$ contains all zeros with imaginary part in \[[a-\log\log(Na^2)-3, b + \log\log(Nb^2)+ 3],\] then \eqref{e:elliptic} holds.
\end{theorem}

\begin{proof}
We first bound the coefficients $c(p^m)$. For $p\mid N$ we have $\abs{c(p^m)} = \abs{\eps(p^m)}\leq 1$. Otherwise, there exists a quadratic integer $\alpha_p$ by Hasse's theorem, satisfying $a_p = 2\Ren(\alpha_p)$ and $\alpha_p \overline{\alpha_p}=p$. Consequently, the Euler factor at $p$ factors into 
\[
(1-\alpha_p p^{-s})^{-1} (1-\overline{\alpha_p} p^{-s})^{-1},
\]
and so we have $\abs{c(p^m)} \leq 2p^{m/2}.$ Therefore, (L3) holds with $\sigma_1= 3/2$ and $C=2$, so by Lemma \ref{l:w_f} the sum on the first line of \eqref{e:elliptic} differs at most by
\[
\frac{16}{\pi} \frac{30^{\frac{1-\pi}{2}}}{\pi-1} < 0.07
\]
from $w_f(f_{a,b,\pi})$.

With the choice $B_k=0.8$, as for the complex places in the proof of Theorem \ref{t:hecke-l-series}, we see from Lemma \ref{l:w_8} and Lemma \ref{l:w8_rem} that the expression on the second line of \eqref{e:elliptic} differs at most by $\mcE(a)+\mcE(b)$ from $w_\infty(f_{a,b,\pi}$). Thus, the first assertion concerning the completeness of the list $\mcL$ follows from Theorem \ref{t:fundamental-inequality}.

For the proof of the second part let $C(X) = \log\log(NX^2) + 3$. As in the proof of Theorem \ref{t:hecke-l-series}, it is again sufficient to show that the inequality
\begin{equation}\label{e:elliptic-converse-to-show}
\sum_{\gamma \notin [a-C(a),b+C(b)]} n_\rho\hat  f_{a,b,\pi}(\gamma) + 0.07 + 2(\mathcal E(a) + \mathcal E(b)) \leq 0.42
\end{equation}
holds. As in the proof of Theorem \ref{t:hecke-l-series} we have
\begin{equation}\label{e:zsum-elliptic-upper}
\sum_{\gamma>X} n_\rho\hat  f_{a,b,\pi}(\gamma) \leq \frac2\pi e^{-C(b)} \Ren \frac{\Lambda'}{\Lambda}(E,2 + iX).
\end{equation}
For the finite part we use the bound $\abs{\frac{L'}{L}(E,2)} \leq 2\frac{\zeta'}{\zeta}(3/2)\leq 3.02$. For the gamma factor we use the Stirling formula \eqref{e:stirling} again. From \eqref{e:elliptic-converse-lower} we see that $C(b)\leq \abs{b}/2$, so we have $\abs{X}\geq 35$ and hence obtain the bound
\[
\Ren\frac{\Gamma'}{\Gamma}(2+iX) \leq \log(\abs X) + 0.06.
\]
Therefore, the right hand side of \eqref{e:zsum-elliptic-upper} is bounded by
\[
\frac2\pi \frac{ e^{-3}}{\log (N b^2)}\left(\half \log (N b^2) + \frac{C(b)}{\abs{b}} + 0.06 + 3.03 - \log(2\pi) \right) < 0.03,
\]
where we also used $\log(N b^2)\geq 8$. Again, the same bound holds for the lower part of the sum, and since we have
\[
0.03+0.03+0.07+ 2(\mathcal E(a)+\mathcal E(b))\leq 0.36,
\]
the assertion follows from \eqref{e:elliptic-converse-to-show}.
\end{proof}

\section{Acknowledgements}
I wish to thank the referee for many helpful comments, which led to a substantial improvement of this paper. I also wish to thank Patricia Klotz for pointing out the mistaken signs in \cite{BFJK13} that were corrected in this paper.

\bibliographystyle{amsalpha}
\bibliography{zero-counting}

\end{document}